\documentclass[12pt]{amsart}

\usepackage{amssymb,amsmath,amsthm}
\usepackage{simplemargins, verbatim}
\usepackage{pinlabel}
\usepackage{graphicx}
\usepackage{enumerate}
\usepackage{dsfont}

\setallmargins{1.5 in}

\renewcommand{\epsilon}{\varepsilon}
\renewcommand{\setminus}{\smallsetminus}

\theoremstyle{plain}
\newtheorem{theorem}{Theorem}
\newtheorem*{theorem*}{Theorem}
\newtheorem{corollary}{Corollary}[section]
\newtheorem{lemma}[corollary]{Lemma}
\newtheorem{proposition}[corollary]{Proposition}

\theoremstyle{definition}
\newtheorem*{definition}{Definition}

\theoremstyle{remark}

\renewcommand{\mathbb}{\mathds}
\newcommand{\Prob}{{\mathbb P}}
 
\newcommand{\E}{{\mathbb E}}
\newcommand{\R}{{\mathbb R}}
\newcommand{\C}{{\mathbb C}}

\newcommand{\Half}{{\mathbb H}}
\newcommand{\Disk}{{\mathbb D}}
\newcommand{\LL}{{\mathbb L}}
\newcommand{\1}{{\mathds{1}}}

\newcommand{\bd}{\partial}

\newcommand{\W}{{\mathcal W}}

\renewcommand{\Re}{{\rm Re}}

\newcommand{\Cov}{{\rm Cov}}
\newcommand{\SLE}{{\mathrm{SLE}}}

\newcommand{\mwhere}{{\ \ \text{where} \ \ }}
\newcommand{\mand}{{\ \ \text{and} \ \  }}

\newcommand{\rd}{{d}}
\newcommand{\dd}{{\ \rd}}

\newcommand{\ds}{{\dd s}}
\newcommand{\dt}{{\dd t}}

\begin{document}

\title{The parafermionic observable in $\SLE$}
\author{Brent Morehouse Werness}

\maketitle

\begin{abstract}
	The parafermionic observable has recently been used by number of authors to study discrete models, believed to be conformally invariant and to prove convergence results for these processes to $\SLE$ \cite{OffCrit,buzios,Connective,energy,cardy2,cardy1,cardy0,towards,ising1,discrete}.  We provide a definition for a one parameter family of continuum versions of the paraferminonic observable for $\SLE$, which takes the form of a normalized limit of expressions identical to the discrete definition.  We then show the limit defining the observable exists, compute the value of the observable up to a finite multiplicative constant, and prove this constant is non-zero for a wide range of $\kappa$.  Finally, we show our observable for $\SLE$ becomes a holomorphic function for a particular choice of the parameter, which helps illuminate a fundamental property of the discrete observable.

\end{abstract}

\section{Introduction}

The paraferminonic observable is a tool which has been used successfully to study discrete models in both the physics \cite{cardy2,cardy1,cardy0} and mathematics \cite{OffCrit,buziosD,Connective,energy,towards,ising1,discrete} literature.  In particular, the program by Smirnov and coauthors to prove the convergence of various discrete processes to $\SLE_\kappa$ \cite{buziosD,towards,discrete} is of great interest. This program has already produced a proof of the convergence of Ising random cluster model boundaries and Ising cluster boundaries to $\SLE_{16/3}$ and $\SLE_3$, respectively \cite{ising1}.  Recent work by Duminil-Copin and Smirnov in \cite{Connective} use these same techniques to rigorously compute the connective constant for self-avoiding walk and provide an approach to showing planar self-avoiding walk converges in the scaling limit to $\SLE_{8/3}$.

Central to this program is an object known as the parafermionic observable, which is a one-parameter family, indexed by the \emph{spin} $\sigma$, of observables defined by the tuning number of the discrete curves.  The fundamental property of this observable is that one may often find a value of $\sigma$ for which the observable is discretely holomorphic in the end-points of the curve, which allows one to deduce conformal invariance.  The reason for the existence of such a $\sigma$ is not well-understood.

In this paper, we provide a definition for a similar observable for radial $\SLE_\kappa$, derived from the discrete definition as reviewed in Section \ref{discrete}. We then prove the existence of the observable, compute the value of the observable up to an overall multiplicative constant, and finally show this multiplicative constant is non-zero for a large range of $\kappa$ by producing an explicit lower bound.  In this setting, we may then compute in the continuum the proper value of $\sigma$ to produce the holomorphic observable, providing some understanding of the value of $\sigma$ observed in the discrete setting.

We note that a related study was conducted, from a physical perspective, for chordal $\SLE$ using a different definition for the continuum parafermionic observable based on the probability that an $\SLE$ curve passes to the left of a point \cite{cardy0}, and from the point of view of conformal field theory \cite{cardyTalk}.

\subsection{The discrete parafermionic observable}\label{discrete}

Before discussing the continuum generalization of the parafermionic observable, we first review the discrete observable through the example of the self-avoiding walk (SAW) as used in \cite{Connective}.  Let $\LL$ denote the hexagonal lattice  in $\C$ with unit distance between adjacent points.  Fix a bounded domain $\Omega \subseteq \C$ and let $V(\Omega)$ be the set of vertices of $\LL \cap \Omega$.  We turn $V(\Omega)$ into a graph by adding every half-edge adjacent to any vertex $V(\Omega)$ and mid-edge vertices along every edge.  See Figure \ref{disFig} for an illustration of this setup.

\begin{figure}[htb]
	\labellist
		\pinlabel $z$ [bl] at 97.5 86
		\pinlabel $w$ [l] at 165 98
		\pinlabel $\Omega$ at 120 30
		\pinlabel $\gamma$ [r] at 80 60
	\endlabellist
	\centering
\includegraphics[width=0.8\textwidth]{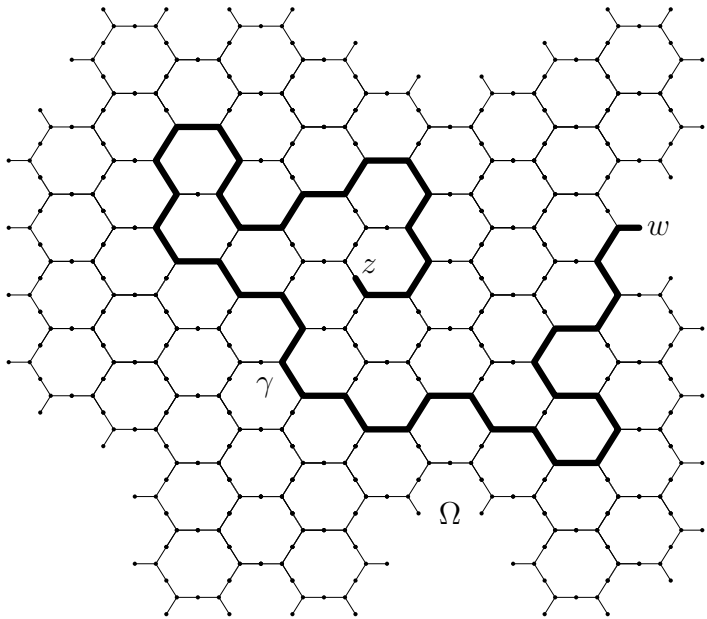}
\caption{An example of a discrete curve $\gamma$ from $w$ to $z$ in the domain $\Omega$ with all vertices and mid-edges illustrated. In this instance the turning number of the curve is $-7\pi/3$.}\label{disFig}
\end{figure} 

Fix mid-edges $z$ in $\Omega$ and $w$ on the boundary of $\Omega$.  Given a parameter $x$, we consider the measure on self avoiding curves starting at $w$ and ending at $z$ on the discretized version on $\Omega$ which weighs a curve $\gamma$ by the weight $x^{\ell(\gamma)}$ where $\ell(\gamma)$ is the number of steps in $\gamma$.  Thus,
\[
F(z) = \E[e^{-i\sigma \W_\gamma(w,z)}] = \sum_{\gamma \subset \Omega: w \rightarrow z} e^{-i\sigma \W_\gamma(w,z)}x^{\ell(\gamma)},
\]
where the expectation notation is being used even though the measure need not be a probability measure.

The key result of \cite{Connective}, for the study of SAW, is that for $x = 1/\sqrt{2+\sqrt{2}}$ and $\sigma = \frac{5}{8}$, $F(z)$ satisfies a relation which can be thought of as a (partial) discrete version of the Cauchy-Riemann equation.  Moreover, by considering limits of this observable in carefully chosen domains, one may derive that the connective constant is $\sqrt{2+\sqrt{2}}$.

There are two important observations to make about this definition.  First, the expectation needs to be taken weighted by the total mass of the curves under the discrete measure and should not be taken under the normalized measure.  Second, a key feature of this observable is the existence of a value of $\sigma$ which makes $F(z)$ holomorphic.  These observations help motivate our study of the $\SLE$ parafermionic observable.

\subsection{Lifting the turning number to the continuum}
In the above definition, we required the turning number of the discrete curve---a property only well-defined for piecewise $C^1$ curves (note that an arbitrary choice must be made at corners; in our case the correct choice is clear).  Indeed the turning number of $\SLE_\kappa$ is inherently ill-defined, and thus we must find an alternative quantity which survives in the continuum.

Examine the discrete curve in Figure \ref{disFig}.  We consider this curve parametrized continuously to be a curve $\gamma:[0,\infty] \rightarrow \C$ with $\gamma(0) = w$ and $\gamma(\infty) = z$.  By the construction of the boundary, the first step of the curve is taken normal to the boundary of the domain.  The last step is taken in the direction of $\arg(\gamma'(\infty))$, which is
\[
\lim_{t\rightarrow \infty} \arg(z-\gamma(t)).
\]
Thus, the turning number is 
\[
\lim_{t\rightarrow \infty} \arg(z-\gamma(t)) - \arg(\mathbf{n}_\Omega(w)) + 2\pi k
\]
for some integer $k$, where $\mathbf{n}_\Omega(w)$ is the inward facing normal.

\begin{figure}[htb]
	\labellist
		\pinlabel $z$ [bl] at 97.5 86
		\pinlabel $w$ [l] at 165 98
		\pinlabel $\Omega$ at 120 30
		\pinlabel $\gamma$ [r] at 80 60
	\endlabellist
	\centering
\includegraphics[width=0.8\textwidth]{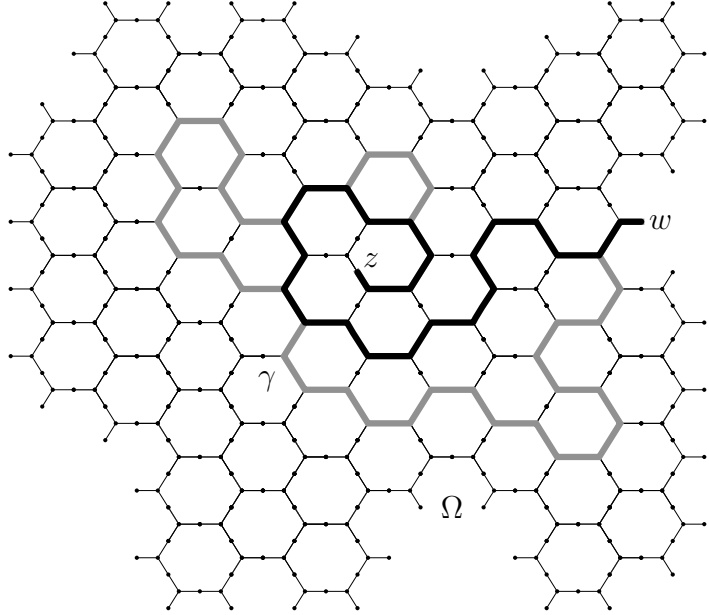}
\caption{The curve $\gamma$ of Figure \ref{disFig} is shown here in gray along with a discretely homotopy equivalent curve in black which illustrates that the winding number dictates the correct multiple of $2\pi$.}\label{disFig2}
\end{figure}

To identify the value of $k$, we exploit the regular homotopy invariance of turning number, which in the discrete setting simply means we may reroute the curve around hexagons as long as doing so does not create any self intersections or move either end.  By performing such operations, we may remove any winding which occurs around points other than $z$ without changing the turning number and hence reduce it to a curve as shown in Figure \ref{disFig2}.  In this case, the correct multiple of $2\pi$ is precisely the winding number of $\gamma$ around $z$. Taking the version of $\arg(z-\gamma(t))$ which is continuous in $t$, we thus have $k=0$ in the above formula.

There is a subtlety in the choice of $\arg(\gamma(0)-z)$ in the above definition.  In particular, care must be taken in situations as in Figure \ref{badDomainFig}, where the domain itself adds turning to the curve by winding around $w$.  As we work mainly in $\Disk$, the unit disk, this is not an issue.  

\begin{figure}[htb]
	\labellist
		\pinlabel $z$ [r] at 40 80
		\pinlabel $w$ [t] at 120 50
	\endlabellist
	\centering
\includegraphics[width=0.8\textwidth]{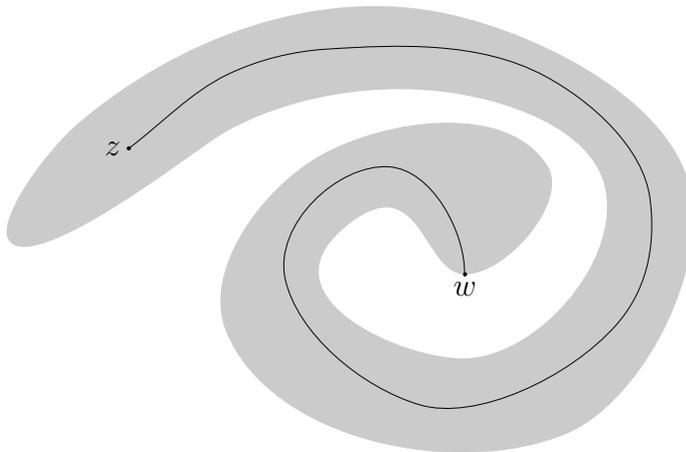}
\caption{A domain with non-trivial winding about $w$}\label{badDomainFig}
\end{figure}

It will be slightly more convenient to consider $\arg(\gamma(t)-z)$, which differs from the previous choice by $\pi$.  Thus, if we take $\mathbf{n}_\Omega(w)$ to be the outward facing normal, we find the turning number is
\[
\lim_{t\rightarrow \infty} \arg(\gamma(t)-z) - \arg(\mathbf{n}_\Omega(w)).
\]

There are two things to note here.  First, there is no ``first step'' of an $\SLE$ along the inwards facing normal. However, as this is the case in the discrete model, it makes sense to adopt this convention.  Second, as an $\SLE$ curve approaches the terminal point, the limit in the above formula will not exist---it needs to be properly normalized in the $t \rightarrow \infty$ limit.  With these points in mind, we make the following definition.

\begin{definition}
	Given an arbitrary non-crossing curve $\gamma:[0,\infty] \rightarrow D$ from $w \in \bd D$ to $z \in D$ where $\bd D$ is sufficiently smooth so that the outward facing normal $\mathbf{n}_D(w)$ at $w$ is well-defined, then 
	\[
	\W_{\gamma(0,t]}^z = \arg(\gamma(t)-z) - \arg(\mathbf{n}_D(w))
	\]
	where the version of the argument is chosen to be continuous in $t$ and to tend to the correct value as $t$ tends to zero.
\end{definition}

An important feature to note about using a normalized limit of $\W_{\gamma(0,t]}^z$ is that this quantity will transform the same as the discrete version under conformal transformations.  If one were to apply a conformal transformation to the discrete curve, the turning number is increased by the argument of the derivative at the terminal point and decreased by the argument of the derivative at the initial point.  This choice may seem odd from the point of view of conformal field theory since it implies that the observable will be holomorphic at the terminal point and anti-holomorphic at the initial point, however this choice is natural when one attempts to mirror the discrete definition.

\subsection{The continuum parafermionic observable}\label{contpara}

With the notion of turning number generalized, we define the parafermionic observable for $\SLE_\kappa$ by mirroring the discrete construction.  

Here, and in the rest of the paper, we will let $a = 2/\kappa$.  Consider a radial $\SLE_\kappa$ in $\Disk$ from $1$ to some point $z$, parametrized so that the conformal radius of $z$ in $D_t := \Disk \setminus \gamma(-s_0,t]$ is $e^{-2at}$ (the curve starts at some time $-s_0=-s_0(z)$ since the conformal radius of $z$ in $D$ is not always $1$; see Section \ref{defSec} for more detailed discussion).  This is (a linear function of) the standard radial parametrization as was used in Schramm's original work defining $\SLE$ \cite{first}.  Let $\W_{\gamma(-s_0,t]}^z$ be as defined in the previous section.  By analogy with the discrete case, we examine the expression
\[
\E^z[e^{-i\sigma W_{\gamma(-s_0,t]}^z}].
\]

This expectation differs from the discrete parafermionic observable in two important ways.  First, this observable does not take into account the entire curve $\gamma$ from $1$ to $z$, only the curve up to time $t$.  Second, the observable only considers the $\SLE$ probability measure, which is not the correct analogue of the discrete measures encountered in Section \ref{discrete}.

To avoid the first issue, we consider the limit as $t \rightarrow \infty$.  As $t$ grows, the typical winding number about $z$ grows as well, and hence the expectation shrinks towards zero.  In Section \ref{existSec}, we will see the form of normalization needed to avoid triviality of the observable is
\[
\lim_{t \rightarrow \infty} e^{2a\nu t}\;\E^z[e^{-i\sigma \W_{\gamma(-s_0,t]}^z}].
\]
for some explicit $\nu = \nu(\kappa,\sigma)$.

To avoid the second issue, we must consider not the probability measure given by $\SLE_\kappa$, but instead the finite measure given by weighting by the total mass of radial $\SLE_\kappa$ as used by Lawler (see, for example, \cite{parkcity}), which we denote by $C_\Disk(w,z)$ where $w\in\bd \Disk$ is the initial point of the $\SLE_\kappa$ and $z \in \Disk$ is the terminal point of the $\SLE_\kappa$. These weighted measures were defined with the intention of mirroring the properties expected for the total mass of the discrete measures (see \cite{parkcity} for a more detailed exposition).  A precise definition is given in Section \ref{totSec}.

With this final addition in place, we are ready to state our main object of study.
\begin{definition}
	The \emph{parafermionic observable for $\SLE_\kappa$} (in the unit disk from $1$ to $z$) is
	\[
	F(z) = \lim_{t \rightarrow \infty} C_\Disk(1,z)  \;e^{2a\nu t}\;\E^z[e^{-i\sigma \W_{\gamma(0,t]}^z}],
	\]
	where $\nu = \nu(\kappa, \sigma)$ is chosen so the limit exists.
\end{definition}

\subsection{Structure of the paper}

The paper is arranged as follows.  

In Section \ref{SLESec}, we provide a brief definition of $\SLE$ and the total mass for radial $\SLE$.  We then introduce a version of the reverse Loewner equation, our main tool for studying the continuum version of the parafermionic observable, which is particularly well suited to our task since the turning number is very easily expressed in terms of the flow.

In Section \ref{existSec}, we use this Loewner equation to show the $\SLE$ parafermionic observable exists and compute it up to a multiplicative constant (which at this stage of the paper may be zero, but is guaranteed finite).  In particular, we show in Theorem \ref{limitFunction} that if $\nu$ is taken to be $\kappa\sigma^2/2$, then the limit defining $F(z)$ exists, and
\[
F(z) = |1-z|^{-2(b-\sigma)}(1-|z|^2)^{b-\tilde b - \nu}(1-z)^{-2\sigma}F(0)
\]
where $b = \frac{6-\kappa}{2\kappa}$ is the boundary scaling exponent for $\SLE_\kappa$, and $\tilde b = (\frac{\kappa - 2}{4})b$ is the interior scaling exponent for $\SLE_\kappa$.  In the process of demonstrating this, we also produce a conformal covariance rule for $F(z)$, which may be found in Theorem \ref{covRule}.

In Section \ref{confSec} we show in Theorem \ref{holTheorem} that $F(z)$ is a holomorphic function of $z$ for precisely the choice $\sigma = b$, which agrees with the conjectured limits for the discrete examples discussed in \cite{buziosD,towards,discrete}.

Finally, in Section \ref{nonzeroSec}, we provide conditions on the non-triviality of the observable by exploiting a particularly convenient reparametrization of the reverse Loewner equation.  In Theorem \ref{nonzeroTheorem}, we provide a general condition to show non-triviality of $F(z)$ which holds even in the non-holomorphic case.  As we are most interested in the holomophic case, we then specialize it in Theorem \ref{confRange} to show the holomorphic observable is non-zero for $\kappa$ approximately in the range $(2.47\ldots,7.99\ldots)$, which is defined by a pair of roots to explicit transcendental equations.

\section{$\SLE$ preliminaries}\label{SLESec}

We review a few preliminaries about radial $\SLE_\kappa$.  We refer the reader to \cite{Lbook,parkcity,Werner} for more complete introductions to the field.

\subsection{Basic definitions}\label{defSec}

Fix a simply connected domain $D$, $w \in \bd D$, and $z \in D$.  We now describe radial $\SLE_\kappa$ from $w$ to $z$ in $D$.  

Radial $\SLE_\kappa$ is the unique family of probability measures, $\mu_D(w,z)$ on non-crossing curves from $w$ to $z$ contained in $D$, viewed moduluo reparametrization, which satisfy the following two properties:
\begin{description}
	\item[{\bf Conformal Invariance}] Given a conformal transformation (by which we mean a conformal map with a conformal inverse) $f:D \rightarrow f(D)$, we have $\mu_{f(D)}(f(w),f(z)) = f \circ \mu_D(w,z)$, where the right hand side is the push-forward of $\mu_D(w,z)$ along $f$.
	\item[{\bf Domain Markov}]  Given $\gamma_t = \gamma(0,t]$, an initial segment of $\gamma$, then conditional distribution of $\mu_D(w,z)$ on the remainder of the curve is $\mu_{D \setminus \gamma_t}(\gamma(t),z)$.
\end{description}

These two properties are obtained by examining discrete models, such as the self-avoiding walk.  The study of $\SLE$ was initiated by Oded Schramm in \cite{first}, in which he showed the above properties may be used to derive a much more technically convenient definition, which we now give.  Throughout we let $a := 2/\kappa$, and most equations are written in term of $a$ rather than $\kappa$.

By conformal invariance, it suffices to only define radial $\SLE_\kappa$ from $1$ to $0$ in $\Disk$.  Let $g_t$ be the solution to the initial value problem in $\Disk$ given by
\[
\bd_t g_t(z) = 2a g_t(z) \frac{e^{2iB_t}+g_t(z)}{e^{2iB_t}-g_t(z)}, \quad g_0(z) = z
\]
where $B_t$ is a standard Brownian motion, called the \emph{driving function}.  This is known as the \emph{radial Loewner equation}.  The map $g_t$ defines a conformal transformation from a simply connected domain $D_t$ to $\Disk$.  Moreover, it is shown in \cite{RS} (for $\kappa \ne 8$) and \cite{loop} (for $\kappa = 8$) that $D_t$ may be written as the connected component of $\Disk \setminus \gamma(0,t]$ containing zero with $\gamma(t)$ given by
\[
\gamma(t) = \lim_{\delta \downarrow 0} g_t^{-1}((1-\delta)e^{2iB_t}).
\] 
This curve is \emph{radial $\SLE_\kappa$ from $1$ to $0$ in $\Disk$}.

Given a simply connected domain $D$ and a point $z \in D$, the \emph{conformal radius} is a measure of distance from $z$ to the boundary defined as $1/|f'(z)|$ where $f$ is a conformal transformation of $D$ to $\Disk$ sending $z$ to $0$.  By differentiating in $z$, we see $\bd_t g_t'(0) = 2a g_t'(0)$ and hence $g_t'(0) = e^{2at}$ which is to say the conformal radius of $0$ in $D_t$ is $e^{-2at}$.

Thus, we have defined radial $\SLE_\kappa$ in $\Disk$ from $1$ to $0$ parametrized so the conformal radius of $0$ in $D_t$ is $e^{-2at}$.  We may use this to define the curve in other domains by conformal invariance, but due to our application, we need to be careful in our choice of parametrization.  We want that for any domain $D$, our $\SLE_\kappa$ curves from $w \in \bd D$ to $z \in D$ are parametrized so the conformal radius of $z$ in $D \setminus \gamma(-s_0,t]$ is $e^{-2at}$.   In particular if $f:\Disk \rightarrow D$ with $f(1) = w$ and $f(0) = z$, we reparametrize so that
\[
(f \circ \gamma)(t) = f(\gamma(s_t)) \mwhere s_t := t+\frac{1}{2a}\log|f'(0)|.
\]
This makes the conformal radius of $f(0)$ in $D\setminus(f\circ\gamma)(-s_0,t]$ equal to $e^{-2at}$.  This has the slight inconvenience that our $\SLE$ curves will often be parametrized to start at time $0$ but instead at time $-s_0 = -\frac{1}{2a}\log|f'(0)|$.

\subsection{Total mass of radial $\SLE$}\label{totSec}
As stated in Section \ref{contpara}, we are not primarily interested in $\SLE$ as a probability measure, as defined in the previous section, but instead as a finite measure weighted by what is known as the $\SLE$ total mass, which we denote by $C_D(w,z)$.  

Throughout, we let $b = (3a-1)/2$ denote the \emph{boundary scaling exponent}, and $\tilde b = \frac{1}{2}(\frac{1}{a} - 1)b$ denote the \emph{interior scaling exponent}.  Then, the \emph{total mass for radial $\SLE_\kappa$} is defined as follows. For convenience, fix $C_\Disk(1,0) = 1$. Then, using the covariance rule derived in \cite{parkcity}, we get for any conformal map $g:\Disk \rightarrow g(\Disk)$ with well-defined derivative at $w \in \bd\Disk$, we have
\[
C_{\Disk}(w,z) = |g'(w)|^b\cdot|g'(z)|^{\tilde b}\cdot C_{g(\Disk)}(g(w),g(z)).
\]

In this work, we are particularly interested in the case $C_\Disk(1,z)$.  Consider the conformal map $f_z : \Disk \rightarrow \Disk$ which is the unique conformal automorphism of $\Disk$ for which $f_z(z) = 0$ and $f_z(1) = 1$.  As this function is useful several times in our paper, we collect a few properties of it in a lemma.

\begin{lemma}\label{fzLemma}
	For $f_z: \Disk \rightarrow \Disk$ as defined above, we have the following statements:
	\begin{align*}
		f_z(w) &= \frac{1-\bar z}{1-z} \cdot \frac{w-z}{1-\bar z w}, \\
		f'_z(w) &= \frac{1-\bar z}{1-z} \cdot \frac{1-|z|^2}{(1-\bar z w)^2}, \\
		|f'_z(1)| &= \frac{1-|z|^2}{|1-z|^2}, \\
		|f'_z(z)| &= \frac{1}{1-|z|^2}, \\
		\arg(f'_z(1)) & = 0, \textrm{ and} \\
		\arg(f'_z(z)) &= \arg\Big(\frac{1-\bar z}{1-z}\Big) = -2\arg(1-z).
	\end{align*}
\end{lemma}
\begin{proof}
	The particular form of $f_z$ follows since every conformal automorphism of the disk, $f$, is of the form
	\[
	f(w) = e^{i\theta}\frac{w-z}{1-\bar z w}
	\]
	for some $z \in \Disk$ and $\theta \in \R$ (see, for example, \cite{stein}).  With the explicit form available, the rest is computation.
\end{proof}

By applying the conformal covariance rule for the total mass, we see
\[
C_\Disk(1,z) = |f_z'(1)|^b\cdot|f_z'(z)|^{\tilde b}\cdot C_{\Disk}(1,0) = |1-z|^{-2b}(1-|z|^2)^{b-\tilde b}.
\]

\subsection{Alternate forms of the Loewner equation and the turning number}\label{altSec}

To prove our theorem, we work with a lifted form of the radial Loewner equation and a corresponding reversed equation.  Let $g_t$ be the solution to the initial value problem
\[
\bd_t g_t(z) = 2a g_t(z) \frac{e^{2iB_t}+g_t(z)}{e^{2iB_t}-g_t(z)}, \quad g_0(z) = z
\]
as in the previous section.

We lift this equation to the upper half plane via the same method as \cite{Lnew}.  Define a function $h_t(z)$ by requiring
\[
g_t(e^{2iz}) = e^{2ih_t(z)}.
\]
This does not define $h_t(z)$ uniquely, but by translating the Loewner equation to an equation on $h_t(z)$, we have
\[
\bd_t h_t(z) = a \; \cot(h_t(z) - B_t)
\]
as long as we choose $h_t(z)$ to be continuous in $t$.  We can make $h_t(z)$ uniquely defined by choosing $h_0(z) = z$ as the initial conditions.

Consider $h_t^{-1}$.  By the definition of the curve generated by the Loewner chain, we see that
\[
\gamma(t) = \lim_{\delta \downarrow 0} g_t^{-1}((1-\delta)e^{2iB_t}) = \lim_{\delta \downarrow 0} e^{2ih_t^{-1}(B_t+i\delta)}
\]
where the limit exists for all $t$ with probability one.  Thus if one is interested in the winding number, we need to consider
\[
\W_{\gamma(0,t]}^0 = \arg(\gamma(t)) - \arg(\mathbf{n}_\Disk(1)) = \lim_{\delta \downarrow 0} 2\Re(h_t^{-1}(B_t+i\delta))
\]
where each step must interpreted as the version which is continuous in $t$.  As the right hand side is already continuous in $t$ with the correct value of $0$ at $t = 0$, we need not worry further.

When one wishes to study the behavior of the inverse of a Loewner map, it has proven to be useful to consider the reverse Loewner flow (see, for example, \cite{RS,multi}).  In our setting, this is a process $Z_t$ satisfying
\[
\rd Z_t(z) = -a\cot(Z_t(z))\dt + \rd B_t, \quad Z_0(z) = z.
\]
This is related to $h_t^{-1}$ as follows.

\begin{lemma}\label{distLem}
For a fixed $t > 0$, the random conformal maps $h_t^{-1}(B_t + \cdot)$ and $Z_t(\cdot)$ are equal in law.
\end{lemma}
The proof of this lemma is identical to the equivalent statement for the chordal case, which may be found, for example, in \cite[Lemma 5.2]{multi}, and is thus omitted.  It is important to note this only holds for any \emph{fixed} $t>0$ and not on the process level.  

\begin{figure}[htb]
	\labellist
	\small
	\pinlabel $i\delta$ [b] at 45 15

    \pinlabel $Z_t$ [b] at 100 75

	\pinlabel $\Theta_t(i\delta)$ [b] at 130 70

	\tiny
    \pinlabel $-\pi$ at 15 5
	\pinlabel $0$ at 45 5
	\pinlabel $\pi$ at 75 5
	
	\pinlabel $-\pi$ at 125 5
	\pinlabel $0$ at 155 5
	\pinlabel $\pi$ at 185 5
	\endlabellist
	\centering
\includegraphics[width=0.8\textwidth]{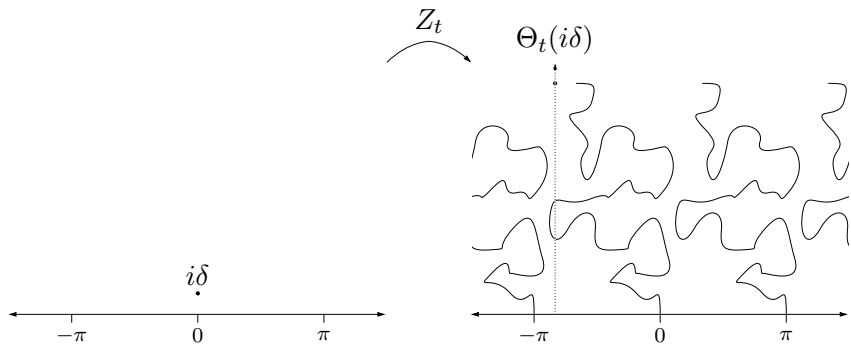}
\caption{A diagram illustrating the lifted reverse radial Loewner equation $Z_t$.  This is a conformal map from $\Half$ to a subdomain of $\Half$ slit by repeated translates of curves growing out of $0$ (and its translates).  While convenient for comparing with the turning number, this makes the growth of $\Theta_t$ erratic.}\label{ThetaGrowFig}
\end{figure}

For simplicity of notation, let us write $Z_t(z) = \Theta_t(z) + i R_t(z)$.  Thus we have that, in law,
\[
\W_{\gamma(0,t]}^0 = \lim_{\delta \downarrow 0} 2\Theta_t(i\delta).
\]
In this setup, to analyze the turning number we need only analyze the real part of the reverse flow.

By the definition of $\cot(z)$ we have
\[
\cot(x+iy) = i \frac{e^{i(x+iy)} + e^{-i(x+iy)}}{e^{i(x+iy)} - e^{-i(x+iy)}} = \frac{\sin(2x)}{\cosh(2y)-\cos(2x)} - i\frac{\sinh(2y)}{\cosh(2y)-\cos(2x)}
\]
and thus by applying It\^o's formula
\begin{align*}
\rd \Theta_t & = -a \frac{\sin(2\Theta_t)}{\cosh(2R_t)-\cos(2\Theta_t)} \dt + \rd B_t, \\
\rd R_t & = a \frac{\sinh(2R_t)}{\cosh(2R_t)-\cos(2\Theta_t)} \dt.
\end{align*}

We now split $\Theta_t$ into two parts by defining $T_t$ so that $\Theta_t = T_t + B_t$.  This choice is motivated by the following geometric interpretation of the reverse flow, which is fundamental to the intuition behind the work that follows.

Consider the growth of $\Theta_t$.  This should be thought of as the total motion of the lifted $\SLE_\kappa$ curve in the real direction, which, as seen above, corresponds to the turning number.  Its value should become large, and thus its variance hard to control.  Figure \ref{ThetaGrowFig} illustrates the motion of $\Theta_t$.  

However, now consider $T_t = \Theta_t - B_t$.  This translation has the effect of growing the curves from translates of $B_t$ rather than from translates of $0$.  In this case, the majority of the change in argument has been shifted to change in location of the base of growth of the curve, whereas the tip, which is traced by $T_t$, grows mostly vertically.  Indeed, by our definition, we have that
\[
\rd T_t = -a \frac{\sin(2\Theta_t)}{\cosh(2R_t)-\cos(2\Theta_t)} \dt.
\]
which implies the rate of displacement of $T_t$ decays exponentially in $R_t$.  The exponential decay hints at what we prove rigorously in the next section: $T_t$ has a limit as $t\rightarrow \infty$.    Figure \ref{TGrowFig} illustrates the motion of $T_t$.

\begin{figure}[htb]
	\small
	\labellist
	\pinlabel $i\delta$ [b] at 45 15

    \pinlabel $Z_t-B_t$ [b] at 100 75

	\pinlabel $T_t(i\delta)$ [b] at 155 70
	\pinlabel $B_t$ at 180 5

	\tiny
    \pinlabel $-\pi$ at 15 5
	\pinlabel $0$ at 45 5
	\pinlabel $\pi$ at 75 5
	
	\pinlabel $-\pi$ at 125 5
	\pinlabel $0$ at 155 5
	\pinlabel $\pi$ at 185 5
	\endlabellist
	\centering
\includegraphics[width=0.8\textwidth]{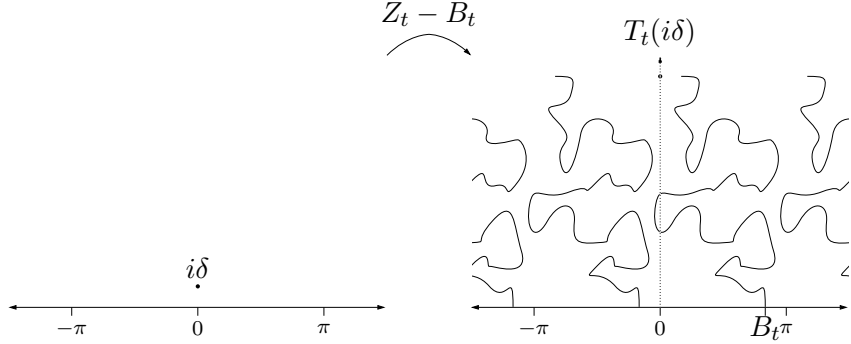}
\caption{This illustrates the growth of the process $T_t$.  As opposed to Figure \ref{ThetaGrowFig}, the new segments of the curve are being added at shifting locations along $\R$.  This has the effect of removing motion of the tip caused by $B_t$, making the tip travel essentially vertically (the effect is exaggerated in this diagram).}\label{TGrowFig} 
\end{figure}

Thus, by splitting $\Theta_t$ into $T_t+B_t$ we have translated understanding the turning number into understanding the growth of a Brownian term and the growth of a term which has a $t\rightarrow \infty$ limit.

\section{The existence of the $\SLE$ parafermionic observable}\label{existSec}
\subsection{The existence of the limit}

We wish to show 
\[
\frac{F(z)}{C_\Disk(1,z)} = \lim_{t \rightarrow \infty} e^{2a\nu t}\;\E^z[e^{-i\sigma \W_{\gamma(0,t]}^z}]
\]
exists when $\nu = \sigma^2/a$.  We do so by first showing it exists when $z = 0$, and then deriving a conformal covariance rule to see the limit exists for all $z$, and moreover computing it as a function of $z$.

Fixing $z = 0$ (and, for simplicity, suppressing it in our notation), we may use the reverse flow facts from above.  Thus we need to show
\[
\lim_{t \rightarrow \infty} e^{2a\nu t}\;\E[e^{-i\sigma \W_{\gamma(0,t]}}] = \lim_{t \rightarrow \infty} \lim_{\delta \downarrow 0} e^{2a\nu t}\;\E[e^{-2i\sigma T_t(i\delta)}e^{-2i\sigma B_t}]
\]
exists (where the location of the $\delta \rightarrow 0$ limit is justified by the dominated convergence theorem).  As a first step we have the following.

\begin{proposition}
	We have
	\[
	\lim_{t \rightarrow \infty} e^{2a\nu t}\;\E[e^{-i\sigma \W_{\gamma(0,t]}}] = \lim_{t \rightarrow \infty} \lim_{\delta \downarrow 0} \E[e^{-2i\sigma T_t(i\delta)}].
	\]
\end{proposition}
\begin{proof} 
	The case $\sigma = 0$ is trivial, and so we need only provide a proof in the case $\sigma \neq 0$.
	By the Cauchy-Schwartz inequality, we have
	\begin{align*}
	\big|\E[e^{-2i\sigma T_t(i\delta)}e^{-2i\sigma B_t}] - \E[e^{-2i\sigma T_t(i\delta)}]\E[e^{-2i\sigma B_t}]\big| & = |\Cov(e^{-2i\sigma T_t(i\delta)},e^{-2i\sigma B_t})| \\
	& \le \left(\E[e^{-4i\sigma B_t}]\right)^{1/2} \\
	& = e^{-4\sigma^2t}
	\end{align*}
	where the second line follows since the variance of $e^{-2i\sigma T_t(i\delta)}$ is no greater than one, and the third line follows by the calculation of the characteristic function of a normal random variable.  Written another way, this states
	\begin{align*}
	\E[e^{-2i\sigma T_t(i\delta)}e^{-2i\sigma B_t}] & = \E[e^{-2i\sigma T_t(i\delta)}]\E[e^{-2i\sigma B_t}] + O(e^{-4\sigma^2t}) \\
	& = e^{-2\sigma^2t}\E[e^{-2i\sigma T_t(i\delta)}] + O(e^{-4\sigma^2t})
	\end{align*} as $t \rightarrow \infty$ where the implicit constant is one, and hence does not depend on $\delta$.
	
	Thus, we have
	\begin{align*}
	e^{2a\nu t}\E[e^{-2i\sigma T_t(i\delta)}e^{-2i\sigma B_t}] & = e^{2a\nu t-2\sigma^2t}\E[e^{-2i\sigma T_t(i\delta)}] + O(e^{2a\nu t-4\sigma^2t}) \\
	& = \E[e^{-2i\sigma T_t(i\delta)}] + O(e^{-2\sigma^2t})
	\end{align*}
	by the definition of $\nu$, where the implicit constant is again one, and hence independent of $\delta$.  As $\sigma^2 > 0$, this gives the desired result.
\end{proof}

Thus, we have reduced the existence of the limit to controlling $T_t(i\delta)$.  We will show
\[
T_\infty := \lim_{t \rightarrow \infty} \lim_{\delta \downarrow 0} T_t(i\delta)
\]
exists with probability one.  This suffices, by the dominated convergence theorem, to show the limit defining $F(0)$ exists.

Consider the process
\[
T_t(0) := \lim_{\delta \downarrow 0} T_t(i\delta).
\]
As mentioned in Section \ref{defSec}, for a fixed time $t \ge 0$, this limit exists due to the existence of the $\SLE_\kappa$ curve.  To study $T_t(0)$ (and related reverse flow processes), we use the following lemma, stating that the limit versions as above all satisfy their equivalent SDEs and ODEs.

\begin{lemma}
	The processes $R_t(0)$, $T_t(0)$, $\Theta_t(0)$, and $Z_t(0)$ all satisfy the ODEs and SDEs satisfied by $R_t(i\delta)$, $T_t(i\delta)$, $\Theta_t(i\delta)$, and $Z_t(i\delta)$ respectively for all $t > 0$ with probability one.
\end{lemma}
\begin{proof}
Fix an arbitrary $t > 0$.  First note $R_t(0) > 0$ with probability one since in distribution $R_t(0) = \log|\gamma(t)|$ and $\gamma(t) \in \Disk$ with probability one.  This is immediate for $\kappa \le 4$ as such $\SLE_\kappa$ curves never hit the boundary with probability one \cite[Theorem 6.1]{RS}.  For $\kappa > 4$, one may see this by observing that the equivalent statement holds for chordal $\SLE$ by scaling and then using absolute continuity of radial $\SLE$ with chordal $\SLE$ (see, for example, \cite{parkcity}).  This implies 
\[
\rd(T_t(0)+iR_t(0)) = \rd (Z_t(0)-B_t) = -a\cot(Z_t(0))\dt
\]
is bounded and continuous in a neighborhood of $T_t(0)+iR_t(0)$ and thus we have
\[
\frac{\rd T_t(0)}{\rd t} = \lim_{\delta \downarrow 0} \frac{\rd T_t(i\delta)}{\rd t} \mand \frac{\rd R_t(0)}{\rd t} = \lim_{\delta \downarrow 0} \frac{\rd R_t(i\delta)}{\rd t}.
\]

This holds with probability one for any fixed $t$, and thus it  holds for all rational $t$.  We may conclude that it holds for all $t>0$ by continuity of the derivative of $T_t(i\delta)$ and $R_t(i\delta)$.

To extend this claim to $Z_t(0)$ and $\Theta_t(0)$, simply note these are sums of $T_t(0)$, $R_t(0)$ and $B_t$.
\end{proof}

We may now prove $T_\infty$ exists.

\begin{proposition}\label{limexistProp}
	$T_\infty := \lim_{t \rightarrow \infty} T_t(0)$ exists with probability one.
\end{proposition}
\begin{proof}
	Let $\tau = \inf\{t > 0 \ \mid \ R_t(0) \ge \log(\sqrt{3})\}$.  We first show this stopping time is finite with probability one.  By Lemma \ref{distLem}, we know
	\[
	Z_t(\cdot) = h_t^{-1}(B_t+\cdot)
	\]
	in distribution.
	Since
	\[
	\gamma(t) = \lim_{\delta \downarrow 0} e^{2ih_t^{-1}(B_t+i\delta)},
	\]
	we obtain
	\[
	\gamma(t) = e^{2i\Theta_t(0) - 2R_t(0)}
	\]
	in distribution.  Since $R_t(0)$ is non-decreasing in $t$ (by the ODE it satisfies),
	\begin{align*}
		\Prob\{\tau < \infty\} & = \lim_{t \rightarrow \infty} \Prob\{\tau < t\}\\
		& = \lim_{t \rightarrow \infty} \Prob\{R_t(0) \ge \log(\sqrt{3})\} \\
		& = \lim_{t \rightarrow \infty} \Prob\Big\{|\gamma(t)| \le \frac{1}{3}\Big\} \\
		& = 1,
	\end{align*}
	where the last line holds since with probability one $\lim_{t\rightarrow \infty} \gamma(t) = 0$ (see \cite{Lnew}).
	
	Since $\tau < \infty$ with probability one, $T_\tau(0)$ is well-defined.  We now use the ODEs satisfied by $R_t(0)$ and $T_t(0)$ to show that for all $t > \tau$, 
	\[
	\left|\frac{\rd T_t(0)}{\rd t}\right| \le \frac{2a}{3}e^{-at},
	\]
	and hence we have $\lim_{t \rightarrow \infty} T_t(0)$ exists with probability one, and indeed cannot differ by more than a constant from $T_\tau(0)$.
	
	First, we bound below the rate of growth of $R_t(0)$ after $\tau$.  For any $t > \tau$, we have
	\begin{align*}
		\frac{\rd R_{t}(0)}{\rd t} & = a \frac{\sinh(2R_t(0))}{\cosh(2R_t(0))-\cos(2\Theta_t(0))} \\
		& = a \frac{1 - e^{-4R_t(0)}}{1+e^{-4R_t(0)}-2e^{-2R_t(0)}\cos(2\Theta_t(0))} \\
		& \ge a \frac{1 - e^{-4R_t(0)}}{1+e^{-4R_t(0)}+2e^{-2R_t(0)}} \\
		& \ge \frac{a}{2}
	\end{align*}
	where the last line holds by our definition of $\tau$.  Thus we have
	\[
	R_{\tau+t}(0) \ge \log(\sqrt{3}) + \frac{a}{2}t.
	\]
	
	We now turn our attention to $T_t(0)$.  For any $t > \tau$, we have
	\begin{align*}
		\left|\frac{\rd T_{\tau+t}(0)}{\rd t}\right| & = \left|a \frac{\sin(2\Theta_{\tau+t}(0))}{\cosh(2R_{\tau+t}(0))-\cos(2\Theta_{\tau+t}(0))} \right| \\
		& \le \frac{2a}{e^{2R_{\tau+t}(0)}+e^{-2R_{\tau+t}(0)}-\cos(2\Theta_{\tau+t}(0))} \\
		& \le 2ae^{-2R_{\tau+t}(0)} \\
		& \le 2ae^{-2(\log(\sqrt{3}) + \frac{a}{2}t)} \\
		& = \frac{2a}{3}e^{-at}
	\end{align*} where we have used our bound on $R_{\tau+t}(0)$ in the second to last line.  Thus we have proven the limit exists with probability one.
\end{proof}

Thus, by the dominated convergence theorem, we have completed our proof of the following.
\begin{proposition} 
The limit
\[
F(0) = C_\Disk(1,0) \lim_{t \rightarrow \infty} e^{2a\nu t}\;\E^0[e^{-i\sigma \W_{\gamma(0,t]}^0}]
\]
exists when $\nu = \sigma^2/a$, and moreover
\[
F(0) = \E^0[e^{-2i\sigma T_\infty}],
\]
which implies $0 \le |F(0)| \le 1$.
\end{proposition}

We now extend this to show the limit defining $F(z)$ exists for all $z \in \Disk$.  Consider $f_z$ as defined in Lemma \ref{fzLemma}.  We have
\begin{align*}
F(z) & = C_\Disk(1,z) \lim_{t \rightarrow \infty} e^{2a\nu t}\;\E^z[e^{-i\sigma \W_{\gamma(-s_0,t]}^z}] \\
& = C_\Disk(1,z) \lim_{t \rightarrow \infty} e^{2a\nu t}\;\E^0[e^{-i\sigma \W_{(f_z^{-1} \circ \gamma)(-s_0,t]}^z}]
\end{align*}
by conformal invariance and our choice of parametrization of $(f_z^{-1} \circ \gamma)$.  Thus, we need to understand $\W_{(f_z^{-1} \circ \gamma)(-s_0,t]}^z$.  

Consider any curve $\gamma:[0,\infty] \rightarrow \Disk$ with $\gamma(0) = 1$ and $\gamma(\infty) = 0$ and, to avoid pathologies, $0 \not \in \gamma(0,\infty)$ (as is the case for radial $\SLE$ curves from $1$ to $0$ in $\Disk$). For any conformal transformation $f$ with well-defined derivative at $1$, we consider $f\circ\gamma$ as being reparametrized by
\[
(f \circ \gamma)(t) = f(\gamma(s_t)) \mwhere s_t := t+\frac{1}{2a}\log|f'(0)|
\]
as we have done with our $\SLE$ curves.  By Taylor expanding the logarithm around $z$, we get
\begin{align*}
\W_{(f\circ\gamma)(-s_0,t]}^z & = \arg((f\circ\gamma)(t)-f(0)) - \arg(\mathbf{n}_{f(\Disk)}(f(1))) \\
& = \arg(f(\gamma(s_t))-f(0)) - \arg(f'(1)) \\
& = \arg(f'(0)) + \arg(\gamma(s_t)) - \arg(f'(1)) +O(|\gamma(s_t)|)\\ 
& = \W_{\gamma(0,s_t]}^0 + \arg(f'(0)) - \arg(f'(1)) +O(|\gamma(s_t)|)
\end{align*}
for $|\gamma(s_t)|$ sufficiently small, where the implicit constant depends only on $f$.  Again, care must be taken so the arguments are chosen to be continuous, which in this case is to say $\arg(f'(z))$ is continuous in $\Disk$.  This establishes the following.

\begin{lemma}\label{windLem}
	Let $f:\Disk \rightarrow f(\Disk)$ be a conformal map with $f'(1)$ well-defined.  Then
	\[
	\W_{(f\circ\gamma)(-s_0,t]}^{f(0)} = \W_{\gamma(0,s_t]}^{0} + \arg(f'(0)) - \arg(f'(1)) +O(|\gamma(s_t)|).
	\]
	as $t\rightarrow \infty$ where the implicit constant depends only on $f$.
\end{lemma}

We now return to our specific case. By using Lemma \ref{windLem} and Lemma \ref{fzLemma}, we see that
\begin{align*}
F(z) & = C_\Disk(1,z) \lim_{t \rightarrow \infty} e^{2a\nu t}\;\E^0[e^{-i\sigma \W_{(f_z^{-1} \circ \gamma)(-s_0,t]}^z}] \\
& = C_\Disk(1,z) \lim_{t \rightarrow \infty} e^{2a\nu t}\;\E^0[e^{-i\sigma [\W_{\gamma(0,s_t]}^0 - \arg(f_z'(z)) + \arg(f_z'(1)) +O(|\gamma(s_t)|)]}] \\
& = C_\Disk(1,z) e^{\nu\log|f_z'(z)|} e^{i\sigma\arg(f_z'(z))} \lim_{t\rightarrow \infty} e^{2a\nu s_t}\;\E^0[e^{-i\sigma[\W_{\gamma(0,s_t]}^0  +O(|\gamma(s_t)|)]}] \\
& = C_\Disk(1,z) |f_z'(z)|^{\nu} e^{-2i\sigma\arg(1-z)} \lim_{t\rightarrow \infty} e^{2a\nu s_t}\;\E^0[e^{-i\sigma[\W_{\gamma(0,s_t]}^0  +O(|\gamma(s_t)|)]}] \\
& = C_\Disk(1,z) (1-|z|^2)^{-\nu} e^{-2i\sigma\arg(1-z)} F(0) \\
& = |1-z|^{-2b}(1-|z|^2)^{b-\tilde b-\nu}e^{-2i\sigma\arg(1-z)} F(0)\\
& = |1-z|^{-2(b-\sigma)}(1-|z|^2)^{b-\tilde b - \nu}(1-z)^{-2\sigma}F(0).
\end{align*}
The third to last line is non-trivial to justify; however, the proof follows the same argument used to prove the existence of the limit.  That is: first rewrite $\W_{\gamma(0,s_t]}^0$ in terms of $B_t$ and $T_t$, second use Cauchy-Schwartz to cancel the Brownian term with the $e^{2a\nu s_t}$ term, and third use dominated convergence to move the limits inside the expectation.

This yields the following theorem.
\begin{theorem}\label{limitFunction}
	When $\nu = \sigma^2/a$, the limit defining $F(z)$ exists for all $z$ and moreover
	\[
	F(z) = |1-z|^{-2(b-\sigma)}(1-|z|^2)^{b-\tilde b - \nu}(1-z)^{-2\sigma}F(0).
	\]
\end{theorem}

The same argument above can, with only a change of notation, yield the following conformal covariance rule.
\begin{theorem}\label{covRule}
Given a domain $D$ which is $C^1$ near a boundary point $w \in \bd D$ and a point $z$ in the interior.  Define
\[
F_D(w,z) := C_D(w,z) \lim_{t \rightarrow \infty} e^{2a\nu t}\;\E^z[e^{-i\sigma \W_{\gamma(-s_0,t]}^z}]. 
\] 
Let $f:D \rightarrow f(D)$ be a conformal transformation with $f'(w)$ well-defined, then  
\[
F_{f(D)}(f(w),f(z)) = |f'(z)|^{\sigma-\tilde b-\nu}|f'(w)|^{-b-\sigma}f'(z)^{-\sigma}f'(w)^{\sigma}F_D(w,z).
\]
\end{theorem}

\subsection{Observations on Theorem \ref{limitFunction}, and the conformal observable}\label{confSec}

We wish to make a few observations on the above results. 

First, the result above does not ensure that $F(z)$ is not uniformly zero in all of $\Disk$.  To prove it is not, at least for a range of $\kappa$, we engage in a careful analysis of the reverse flow, which is the topic of the majority of the remainder of the paper.

Second, recall from the discrete setup that we are most interested in the case where $F(z)$ is a holomorphic function of $z$.

\begin{theorem}\label{holTheorem}
	For any $a$, there is a unique choice of $\sigma$ which makes $F(z)$ a holomorphic function of $z$.  In particular, for $\sigma = b$, we have
	\[
	F(z) = (1-z)^{-2b}F(0)
	\]
	and moreover it satisfies the conformal covariance rule (under the same conditions as Theorem \ref{covRule})
	\[
	F_{f(D)}(f(w),f(z)) = f'(z)^{-b} \overline{f'(w)}^{-b}F_D(w,z).
	\]
\end{theorem}
\begin{proof}

Consider the form of $F(z)$ from Theorem \ref{limitFunction}.  It is a product of $F(0)$, some arbitrary complex constant, $(1-z)^{-2\sigma}$, a holomorphic function, and 
\[
|1-z|^{-2(b-\sigma)}(1-|z|^2)^{b-\tilde b -\nu},
\]
a real valued function.  Thus, for $F(z)$ to be holomorphic, the real-valued function must be constant.  By examining the function in a neighborhood of $1$, we obtain the necessary condition $\sigma = b$.  Direct computation shows the function to be constant for this choice.
\end{proof}

This result agrees with what one would expect from the discrete theory (see \cite{buziosD,towards,discrete}).  Indeed if we return to the self-avoiding walk from Section \ref{discrete}, we see the value of $\sigma$ computed here agrees with what one would expect if the self-avoiding walk were to converge to $\SLE_{8/3}$ (both yielding $\sigma = \frac{5}{8}$).

It is also worth stating that all results proven thus far work for all values of $\kappa$, even in the $\kappa \ge 8$ case where the curves are spacefilling (see, \cite[Section 6]{RS} or \cite[Theorem 3.36]{parkcity} for a discussion of the geometric phases of $\SLE_\kappa$).  In the following sections, where we provide conditions to ensure $F(0) \neq 0$, this is no longer the case, and we are only able to provide non-trivial bounds for $\kappa < 8$.

A final comment of note on this computation is to draw attention to very surprising coincidence required for this proof to work: that $b-\tilde b -\nu = 0$ when $\sigma = b$.  In this computation, there is no \textit{a priori} reason that this should occur---it simply does.  Equally intriguing is that the same computation occurs in the recent work of Cardy despite a distinct point of view \cite{cardyTalk}.  This suggests that there might be a deeper interpretation of this relationship than the mere cancelation observed here.  Stated another way: the choice $\sigma = b$ can be seen directly from the definition to be required to ensure anti-holomorphicity in the boundary point, whereas the understanding of why this implies holomorphicity in the interior point remains formal.

\section{Non-triviality of the $\SLE$ parafermionic observable}\label{nonzeroSec}
We wish to show
\[
F(0) = \E[e^{-2i\sigma T_\infty}] = \lim_{t\rightarrow \infty} \lim_{\delta \rightarrow 0} \E[e^{-2i\sigma T_t(i\delta)}]
\]
does not vanish for a large range of $\kappa$.  To do so, we need to analyze the lifted form of the reverse radial Loewner equation from Section \ref{altSec}. We make use of the following method.

\begin{lemma}\label{nonZeroLem}
	For a symmetric random variable $X$, we have
	\[
	\E[e^{-i\sigma X}] \ge 1 - \frac{1}{2} \sigma^2\E[X^2].
	\]
\end{lemma}
\begin{proof}
Since $X$ is symmetric we know the expectation is real and hence
\[
\E[e^{-i\sigma X}] = \E[\cos(\sigma X)] \ge 1-\frac{1}{2}\sigma^2\E[X^2].
\]
\end{proof}
Thus, to bound $F(0)$ away from zero we bound $\E[T_t(i\delta)^2]$ above uniformly in $t$ and $\delta$.  It is much simpler to work with a reparametrized version of this process in which $R_t$ increases linearly.  In the next three subsections, we perform this reparametrizaion, bound the reparametrized process, and finally transfer this bound back to the originally parametrized process.

For notational simplicity, we let $\|X\|_2 := \E[X^2]^{1/2}$.

\subsection{The key time change}

Recall the equation
\[
\rd Z_t(z) = -a\cot(Z_t(z))\dt + \rd B_t.
\]
We deal exclusively with $Z_t(i\delta)$ and thus suppress the explicit $\delta$ dependence.

Recall $Z_t = \Theta_t + iR_t$ where $\Theta_t$ and $R_t$ satisfy
\begin{align*}
\rd \Theta_t & = -a \frac{\sin(2\Theta_t)}{\cosh(2R_t)-\cos(2\Theta_t)} \dt + \rd B_t, \\
\rd R_t & = a \frac{\sinh(2R_t)}{\cosh(2R_t)-\cos(2\Theta_t)} \dt.
\end{align*}

As noted in our proof of the existence of the limit in Proposition \ref{limexistProp}, $R_t$ is increasing and for sufficiently large $R_t$ the rate of increase is $a+O(e^{-R_t})$.  To take advantage of this regularity, we reparametrize this process by a time change $\tau(t)$ so that $R_{\tau(t)} = t + \delta$.  

For a fixed $\delta > 0$, this time change $\tau$ is very well-behaved.  Indeed, it only differs from a linear time change by at most a $\delta$ dependent additive constant.
\begin{lemma}\label{notBadTimeLem}
For any $\delta$ we have
\[
\sinh^{-1}(\sinh(\delta)e^{at}) \le R_t(i\delta) \le \cosh^{-1}(\cosh(\delta)e^{at})
\]
and hence
\[
at + \log(2\sinh(\delta))-\delta \le \tau^{-1}(t) \le at+\log(2\cosh(\delta))-\delta
\]
and
\[
\frac{1}{a}t+\frac{1}{a}(\delta - \log(2\sinh(\delta))) \le \tau(t) \le \frac{1}{a}t+\frac{1}{a}(\delta - \log(2\cosh(\delta))).
\]
\end{lemma}
\begin{proof}
We have
\[
\frac{\rd R_t}{\rd t} = a \frac{\sinh(2R_t)}{\cosh(2R_t)-\cos(2\Theta_t)}
\]
and hence $R_t$ is increasing and
\[
a \frac{\sinh(2R_t)}{\cosh(2R_t)+1} \le \frac{\rd R_t}{\rd t} \le a \frac{\sinh(2R_t)}{\cosh(2R_t)-1}.
\]
The solutions to these ODEs with $R_0 = \delta$ are
\[
\sinh^{-1}(\sinh(\delta)e^{at}) \mand \cosh^{-1}(\cosh(\delta)e^{at})
\]
giving the desired bounds on $R_t$.

Now, recall $R_{\tau(t)} = t+\delta$ by definition of $\tau(t)$, or equivalently $R_t = \tau^{-1}(t)+\delta$.  Combining this with the bounds
\[
\cosh^{-1}(x) \le \log(2x) \mand \sinh^{-1}(x) \ge \log(2x)
\]
gives the desired bounds on $\tau^{-1}(t)$, which may be rearranged to give the bounds on $\tau(t)$.  
\end{proof}

Given any process $A_t$, we adopt the notation $\hat A_t = A_{\tau(t)}$. 
\begin{lemma}\label{reparamLem}
We have
\[
\frac{\rd \tau}{\rd t} = \frac{\cosh(2(t+\delta))-\cos(2\hat\Theta_t)}{a\sinh(2(t+\delta))},
\]
and thus
\[
\rd \hat\Theta_t =  -\frac{\sin(2\hat\Theta_t)}{\sinh(2(t+\delta))} \dt + \left(\frac{\cosh(2(t+\delta))-\cos(2\hat\Theta_t)}{a\sinh(2(t+\delta))}\right)^{1/2} \rd W_t
\]
and
\[
\rd \hat T_t = -\frac{\sin(2\hat\Theta_t)}{\sinh(2(t+\delta))} \dt.
\]
\end{lemma}
\begin{proof}
By definition of $\tau(t)$,
\[
\int_0^{\tau(t)} a \frac{\sinh(2R_s)}{\cosh(2R_s)-\cos(2\Theta_s)} \ds = t,
\]
and thus
\[
\frac{\rd \tau}{\rd t} = \frac{\cosh(2(t+\delta))-\cos(2\hat\Theta_t)}{a\sinh(2(t+\delta))}.
\]

We may apply this time change to $\Theta_t$ to see
\begin{align*}
\rd \hat\Theta_t & = -\frac{\sin(2\hat\Theta_t)}{\sinh(2(t+\delta))} \dt + \rd \hat B_{t} \\
& =  -\frac{\sin(2\hat\Theta_t)}{\sinh(2(t+\delta))} \dt + \left(\frac{\cosh(2(t+\delta))-\cos(2\hat\Theta_t)}{a\sinh(2(t+\delta))}\right)^{1/2} \rd W_t.
\end{align*}
The claimed statement for $T_t$ now follows.
\end{proof}

\subsection{The bound on $\E[\hat T_t^2]$ and $\E[T_t^2]$}

Our argument proceeds as follows. From Lemma \ref{reparamLem}, we know
\[
\rd \hat T_t = -\frac{\sin(2\hat\Theta_t)}{\sinh(2(t+\delta))} \dt.
\]
This ODE is well-behaved for $t$ large due to the exponential decay provided by $\sinh(t+\delta)$.  Thus, we must concentrate on controlling the ODE for small times.  In this case, the only term which can provide help is $\sin(2\hat\Theta_t)$.  Thus, we must bound the growth of $\hat\Theta_t$. We do so by deriving a differential inequality which bounds the variance of $\hat\Theta_t$.  We then may produce a bound on the growth of the variance of $\hat T_t$ for small times $t$ which when combined with the exponential decay gives us a uniform bound the variance of $\hat T_t$ for all $t$ and $\delta$. 

By Lemma \ref{reparamLem}, we have
\[
\rd \hat\Theta_t  =  -\frac{\sin(2\hat\Theta_t)}{\sinh(2(t+\delta))} \dt + \left(\frac{\cosh(2(t+\delta))-\cos(2\hat\Theta_t)}{a\sinh(2(t+\delta))}\right)^{1/2} \rd W_t 
\]
and hence
\begin{align*}
\rd \hat\Theta_t^2 & =  \left[\frac{\cosh(2(t+\delta))-\cos(2\hat\Theta_t)}{a\sinh(2(t+\delta))}-\frac{2\hat\Theta_t\sin(2\hat\Theta_t)}{\sinh(2(t+\delta))}\right] \rd t \\
& \quad + 2\hat\Theta_t\left(\frac{\cosh(2(t+\delta))-\cos(2\hat\Theta_t)}{a\sinh(2(t+\delta))}\right)^{1/2} \rd W_t. 
\end{align*}

Since the drift and diffusion coefficients in the SDE for $\hat\Theta_t$ are uniformly bounded for any fixed $\delta > 0$, we have for any $S > 0$
\[
\int_0^S \E[\hat\Theta_s^2] \ds < \infty
\]
and thus the same property holds for the diffusion coefficient of the SDE for $\hat \Theta_t^2$.  This is sufficient to imply the diffusion component is a martingale (see, for example, \cite[Corollary 3.2.6]{oksendal}), and thus 
\begin{align*}
\frac{\rd}{\rd t}\E[\hat\Theta_t^2] & = \frac{\cosh(2(t+\delta))-\E[\cos(2\hat\Theta_t)]}{a\sinh(2(t+\delta))}-\frac{2\E[\hat\Theta_t\sin(2\hat\Theta_t)]}{\sinh(2(t+\delta))} \\
& = \frac{\cosh(2(t+\delta))-1}{a\sinh(2(t+\delta))}+\frac{\E[1 - \cos(2\hat\Theta_t) - 2a\hat\Theta_t\sin(2\hat\Theta_t)]}{a\sinh(2(t+\delta))} \\
& \le \frac{1}{a}(t+\delta) + \frac{\E[1 - \cos(2\hat\Theta_t) - 2a\hat\Theta_t\sin(2\hat\Theta_t)]}{a\sinh(2(t+\delta))}
\end{align*}
where we have used the inequality
\[
\frac{\cosh(t)-1}{\sinh(t)} \le t.
\]  

To proceed, we use a bound of the form
\[
1 - \cos(2x) - 2a x\sin(2x) \le \beta x^2
\] for some $\beta = \beta(a)$. Such bounds may be readily found, however, since we use different values of $\beta$ for different ranges of $\kappa$, we proceed in the proof with a general value for $\beta$, and then specialize at the end.

Assuming such a bound, we get
\[
\E[1 - \cos(2\hat\Theta_t) - 2a\hat\Theta_t\sin(2\hat\Theta_t)] \le \beta\E[\hat\Theta_t^2]
\]
and thus the differential equation becomes
\[
\frac{\rd}{\rd t}\E[\hat\Theta_t^2] \le \frac{1}{a}(\delta+t) + \frac{\beta\E[\hat\Theta^2_t]}{2a(\delta+t)}.
\]

The solution to this differential inequality with the initial condition $\E[\hat\Theta_t^2] = 0$ is
\[
\E[\hat\Theta_t^2] \le \frac{2}{4a-\beta}\left[(t+\delta)^2 - \delta^{2-\beta/2a}(t+\delta)^{\beta/2a}\right].
\] 

We use this to obtain a bound on $\hat T_t(i\delta)$.  This satisfies
\[
\rd \hat T_t = - \frac{\sin(2\hat\Theta_t)}{\sinh(2(t + \delta))} \dt,
\]
and thus
\[
\rd \hat T_t^2 = - \frac{2\hat T_t\sin(2\hat\Theta_t)}{\sinh(2(t + \delta))} \dt.
\]
By applying Cauchy-Schwartz, we have
\[
\Big|\frac{\rd}{\rd t} \E[\hat T_t^2]\Big| \le \E\Big[\Big|\frac{\rd}{\rd t} \hat T_t^2\Big|\Big] =  \frac{2\E[|\hat T_t\sin(2\hat\Theta_t)|]}{\sinh(2(t + \delta))} \le \frac{2\|\hat T_t\|_2\cdot\|\sin(2\hat\Theta_t)\|_2}{\sinh(2(t + \delta))}
\]
which can be rearranged to
\[
\Big|\frac{\rd}{\rd t} \|\hat T_t\|_2\Big| = \left|\frac{\frac{\rd}{\rd t} \E[\hat T_t^2]}{2\|\hat T_t\|_2}\right| \le \frac{\|\sin(2\hat\Theta_t)\|_2}{\sinh(2(t + \delta))}.
\]

Before proceeding to integrate this to bound $\|\hat T_t\|_2$, we collect a few statements for later use.
\begin{lemma}\label{tblemma}
Let
\[
U_t = \int_0^t \frac{\|\sin(2\hat\Theta_s)\|_2}{\sinh(2(s + \delta))}\ds.
\]
Then
\[
\E\Big[\Big|\frac{\rd}{\rd t} \hat T_t^2\Big|\Big] \le 2 U_t \frac{\rd U_t}{\rd t} = \frac{\rd U_t^2}{\rd t}.
\]
\end{lemma}
\begin{proof}
Integrate the statement prior to the lemma to get $\|\hat T_t\|_2 \le U_t$.  Plugging this into the statement two above gives the result.
\end{proof}

We now integrate to produce the bound on $\|\hat T_t\|_2$.  Since $|\sin(x)| \le 1 \wedge |x|$, we get
\begin{align*}
\Big|\frac{\rd}{\rd t} \|\hat T_t\|_2 \Big|& \le \frac{1 \wedge 2\|\hat\Theta_t\|_2}{\sinh(2(t + \delta))}\\
& = \frac{1}{\sinh(2(t + \delta))} \wedge
\frac{2^{3/2}\left[(t+\delta)^2 - \delta^{2-\beta/2a}(t+\delta)^{\beta/2a}\right]^{1/2}}{(4a-\beta)^{1/2}\sinh(2(t + \delta))}.
\end{align*}

The left term is small large values of $t$, whereas the right term is small for small values of $t$.  We wish to swap the two bounds at $t_0$ chosen so that
\[
\frac{4a-\beta}{8} = (t_0+\delta)^2 - \delta^{2-\beta/2a}(t_0+\delta)^{\beta/2a}.
\]
We may bound this by
\[
t_0 + \delta \le \left(\frac{4a-\beta}{8}\right)^{1/2}.
\]
By switching at this upper bound for $t_0$, rather than at $t_0$ itself, we only hurt our estimate, and hence we get
\begin{align*}
\|\hat T_t \|_2 & \le \int_\delta^{\left(\frac{4a-\beta}{8}\right)^{1/2}} \frac{2^{3/2}\left[s^2-\delta^{2-\beta/2a}s^{\beta/2a}\right]^{1/2}}{(4a-\beta)^{1/2}\sinh(2s)}\ds + 
\int_{\left(\frac{4a-\beta}{8}\right)^{1/2}}^\infty \frac{1}{\sinh(2s)}\ds \\
& \le \int_\delta^{\left(\frac{4a-\beta}{8}\right)^{1/2}} \frac{2^{3/2}s}{(4a-\beta)^{1/2}\sinh(2s)}\ds +  \int_{\left(\frac{4a-\beta}{8}\right)^{1/2}}^\infty \frac{1}{\sinh(2s)}\ds \\
& \le \frac{1}{2} + \int_{\left(\frac{4a-\beta}{8}\right)^{1/2}}^\infty \frac{1}{\sinh(2s)}\ds \\
& = \frac{1}{2} + \frac{1}{2}\int_{\left(\frac{4a-\beta}{2}\right)^{1/2}}^\infty \frac{1}{\sinh(s)}\ds
\end{align*}
where we used the bound $\sinh(s) \ge s$ for $s \ge 0$.  This may be integrated exactly since
\[
\int_{\left(\frac{4a-\beta}{2}\right)^{1/2}}^\infty \frac{1}{\sinh(s)}\ds = -\log\Big(\tanh{\textstyle\sqrt\frac{4a-\beta}{8}}\Big),
\]
yielding the following result.
\begin{proposition}\label{Tbound2}
Let $\hat T_t$ be as above, then
\[
\|\hat T_t \|_2 \le \frac{1}{2} - \frac{1}{2}\log\Big(\tanh{\textstyle\sqrt\frac{4a-\beta}{8}}\Big),
\]
or in the notation of Lemma \ref{tblemma},
\[
U_\infty \le \frac{1}{2} - \frac{1}{2}\log\Big(\tanh{\textstyle\sqrt\frac{4a-\beta}{8}}\Big).
\]
\end{proposition}

This produces the desired bound after reparametrization.  We now show this bound also holds in the original parametrization.

\begin{proposition}\label{unParamBound}
	\[
	\E[T_t^2] \le \Big[\frac{1}{2} - \frac{1}{2}\log\Big(\tanh{\textstyle\sqrt\frac{4a-\beta}{8}}\Big)\Big]^2.
	\]
\end{proposition}
\begin{proof}
First note
\begin{align*}
\E[T_t^2] & = \E[\hat T_{\sigma^{-1}(t)}^2] \\
& = \E\left[\int_0^{\sigma^{-1}(t)}\frac{\rd \hat T_s^2}{\rd s} \ds\right] \\
& = \E\left[\int_0^\infty \frac{\rd \hat T_s^2}{\rd s} \1\{s \le \sigma^{-1}\} \ds\right] \\ 
& \le \int_0^\infty \E\Big[\Big|\frac{\rd \hat T_s^2}{\rd s}\Big|\Big] \ds.
\end{align*}
Using Lemma \ref{tblemma} and Proposition \ref{Tbound2}, we obtain
\begin{align*}
\E[T_t^2] & \le \int_0^\infty \E\Big[\Big|\frac{\rd \hat T_s^2}{\rd s}\Big|\Big] \ds \\
& \le \int_0^\infty \frac{\rd U_s^2}{\rd s} \ds \\
& = U_\infty^2 \le \Big[\frac{1}{2} - \frac{1}{2}\log\Big(\tanh{\textstyle\sqrt\frac{4a-\beta}{8}}\Big)\Big]^2.
\end{align*}
\end{proof}

\subsection{Non-triviality results}

Combining Lemma \ref{nonZeroLem} with Proposition \ref{unParamBound}, we obtain the following result.
\begin{theorem}\label{nonzeroTheorem}
	For any $a$ and $\sigma$, if there exists $\beta > 0$ with
	\[
	1 - \cos(2x) - 2a x\sin(2x) \le \beta x^2
	\]
	and
	\[
	2\sigma^2\Big[\frac{1}{2} - \frac{1}{2}\log\Big(\tanh{\textstyle\sqrt\frac{4a-\beta}{8}}\Big)\Big]^2 < 1,
	\]
	then the associated parafermionic observable is not uniformly zero.
\end{theorem}

\begin{figure}[ht]
	\labellist
		\small
		\pinlabel ${\scriptstyle -1}$ [r] at 10 10
		\pinlabel ${\scriptstyle 0}$ [r] at 10 60
		\pinlabel ${\scriptstyle 1}$ [r] at 10 110
		\pinlabel ${\sigma}$ [r] at 5 60
		\pinlabel ${\scriptstyle 0}$ [t] at 10 10
		\pinlabel ${\scriptstyle 1}$ [t] at 110 10
		\pinlabel ${\scriptstyle \frac{1}{4}}$ [t] at 35 10
		\pinlabel ${\scriptstyle \frac{1}{2}}$ [t] at 60 10
		\pinlabel ${\scriptstyle \frac{3}{4}}$ [t] at 85 10
		\pinlabel ${a}$ [t] at 60 5
	\endlabellist
	\centering
\includegraphics[width=0.6\textwidth]{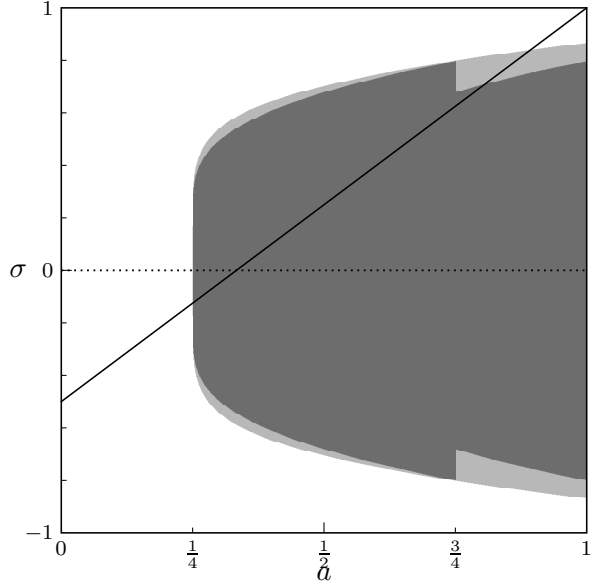}
\caption{A numerical approximation of the set of $a$ and $\sigma$ (shown in gray) for which Theorem \ref{nonzeroTheorem} can prove the parafermionic observable non-trivial.  To produce this figure, the smallest value of $\beta$ which satisfies the first inequality of Theorem \ref{nonzeroTheorem} was numerically approximated for each choice of $\kappa$, and then the corresponding range of valid $\sigma$'s was computed for that combination of $\kappa$ and $\beta$ using the second inequality.  The dark gray subset are those pairs for which non-triviality is established by the choices of $\beta$ in Lemma \ref{functionLem}.  The solid line corresponding to the conformal pairing of $a$ and $\sigma$.}\label{kappaSigmaFig}
\end{figure}

\begin{proof}
\begin{align*}
F(0) & = \lim_{t\rightarrow \infty} \lim_{\delta\downarrow 0} \E[e^{-2i\sigma T_t(i\delta)}] \\
& \ge 1 - 2 \sigma^2 \limsup_{t\rightarrow \infty,\;\delta\downarrow 0} \E[T_t^2(i\delta)] \\
& \ge 1 - 2\sigma^2\Big[\frac{1}{2} - \frac{1}{2}\log\Big(\tanh{\textstyle\sqrt\frac{4a-\beta}{8}}\Big)\Big]^2.
\end{align*}
\end{proof}

Figure \ref{kappaSigmaFig} illustrates a numerical approximation to the range for which this proof is valid.

We are most interested in the $\sigma$ which yields the conformally invariant observable, $\sigma = (3a-1)/2$, and hence we are most interested in finding $\beta$ with
\[
\frac{(3a-1)^2}{8}\Big[1 - \log\Big(\tanh{\textstyle\sqrt\frac{4a-\beta}{8}}\Big)\Big]^2 < 1.
\]
 We provide here explicit choices of $\beta$ which allows us to show the associated parafermionic observables to be non-trivial over a wide range of $\kappa$.  While it may be possible to expand the range slightly with better choices of $\beta$ (particularly for large values of $a$) expanding it beyond $\kappa \in [0,8]$, or indeed to any interval including either end point seems to require a different approach (mainly due to the coarseness of Lemma \ref{nonZeroLem}).

\begin{lemma}\label{functionLem}
	The follow hold:
	\begin{itemize}
		\item for $a \in [1/4,3/4]$, 
		\[
		1 - \cos(2x) - 2a x\sin(2x) \le x^2,
		\]
		\item and for $a \in [1/2,1]$.
		\[
		1 - \cos(2x) - 2a x\sin(2x) \le 2x^2.
		\]
	\end{itemize}
\end{lemma}

Unfortunately, the inequality in Theorem \ref{nonzeroTheorem} may not be solved in closed form forcing us to numerically approximate the range.  Plots of the associated inequalities are given in Figure \ref{goodFig}.  

\begin{figure}[htb]
	\labellist
		\small
		\pinlabel ${\scriptstyle 0.0}$ [r] at 9.5 9.5
		\pinlabel ${\scriptstyle 0.2}$ [r] at 9.5 38.5
		\pinlabel ${\scriptstyle 0.4}$ [r] at 9.5 67
		\pinlabel ${\scriptstyle 0.6}$ [r] at 9.5 96.5
		\pinlabel ${\scriptstyle 0.8}$ [r] at 9.5 124.5
		\pinlabel ${\scriptstyle 1.0}$ [r] at 9.5 153.5
		\pinlabel ${\scriptstyle 0.0}$ [t] at 11.5 8.5
		\pinlabel ${\scriptstyle 0.2}$ [t] at 56 8.5
		\pinlabel ${\scriptstyle 0.4}$ [t] at 101 8.5
		\pinlabel ${\scriptstyle 0.6}$ [t] at 146 8.5
		\pinlabel ${\scriptstyle 0.8}$ [t] at 190.5 8.5
		\pinlabel ${\scriptstyle 1.0}$ [t] at 235 8.5
		\pinlabel ${\scriptstyle a}$ [l] at 240 9.5
	\endlabellist
	\centering
\includegraphics[width=0.7\textwidth]{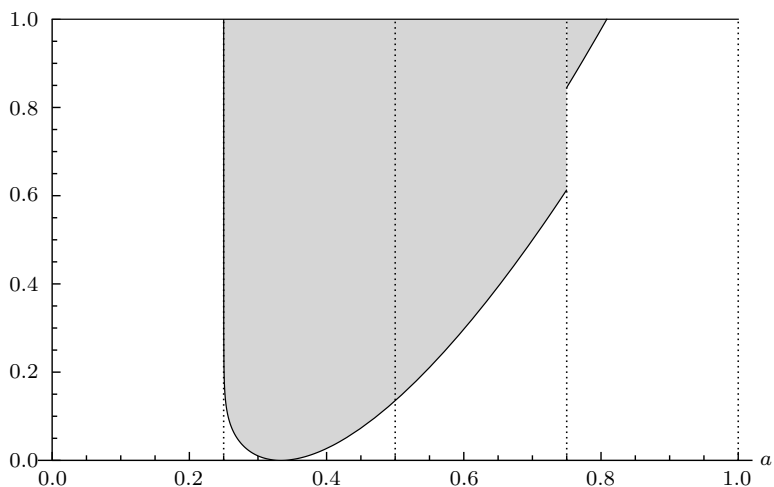}
\caption{A graph showing the ranges over which the inequalities (illustrated as a shaded region) used in the proof of Theorem \ref{confRange} holds.  Vertical dashed lines indicate the ranges where the various choices of $\beta$ hold.}\label{goodFig}
\end{figure}

\begin{theorem}\label{confRange} 
	Let $a_0$ be the solution to
	\[
	\frac{(3a-1)^2}{8}\Big[1 - \log\Big(\tanh{\textstyle\sqrt\frac{4a-1}{8}}\Big)\Big]^2 = 1
	\]
	nearest to $a = 1/4$ and $a_1$ be the solution to
	\[
	\frac{(3a-1)^2}{8}\Big[1 - \log\Big(\tanh{\textstyle\sqrt\frac{2a-1}{4}}\Big)\Big]^2 = 1
	\]
	nearest to $a = 1$.  Then, the conformal parafermionic observable is non-trivial for all $a \in (a_0,a_1)$ which is approximately {\small$(0.2500000022\ldots,\linebreak[1] 0.8084748753\ldots)$}, which corresponds to $\kappa$ in {\small$(2.4737936342\ldots,\linebreak[1] 7.9999999295\ldots)$}\footnote{These values were computed in Mathematica $7$ with a working precision of $20$ digits.}. 
\end{theorem} 

\section*{Acknowledgements}
The author would like to thank Greg Lawler for helpful discussions during the preparation of this paper.

\bibliographystyle{plain}
\bibliography{ParafermionBib}

\begin{thebibliography}{10}

\bibitem{OffCrit}
V.~Beffara and H.~Duminil-Copin.
\newblock Smirnov's fermionic observable away from criticality, 2011.
\newblock arXiv:1010.0526v2.

\bibitem{cardyTalk}
John Cardy.
\newblock Boundary and bulk local operators in conformal field theory and
  {SLE}.
\newblock MSRI Workshop on Statistical Mechanics and Conformal Invariance,
  2012.

\bibitem{buziosD}
H.~Duminil-Copin and S.~Smirnov.
\newblock Conformal invariance of lattice models, 2011.
\newblock arXiv:1109.1549v1.

\bibitem{Connective}
H.~Duminil-Copin and S.~Smirnov.
\newblock The connective constant of the honeycomb lattice equals
  $\sqrt{2+\sqrt2}$, 2011.
\newblock arXiv:1007.0575v2.

\bibitem{energy}
C.~Hongler and S.~Smirnov.
\newblock The energy density in the planar {I}sing model, 2011.
\newblock arXiv:1008.2645v3.

\bibitem{cardy2}
Y.~Ikhlef and J.~Cardy.
\newblock Discretely holomorphic parafermions and integrable loop models.
\newblock {\em J. Phys. A}, 42, 2009.

\bibitem{multi}
F.~Johansson~Viklund and G.~Lawler.
\newblock Almost sure multifractal spectrum for the tip of an {SLE} curve,
  2011.
\newblock arXiv:0911.3983v2.

\bibitem{Lbook}
G.~Lawler.
\newblock {\em Conformally invariant processes in the plane}, volume 114 of
  {\em Mathematical Surveys and Monographs}.
\newblock American Mathematical Society, Providence, RI, 2005.

\bibitem{parkcity}
G.~Lawler.
\newblock Schramm-{L}oewner evolution ({SLE}).
\newblock In {\em Statistical mechanics}, volume~16 of {\em IAS/Park City Math.
  Ser.}, pages 231--295. Amer. Math. Soc., Providence, RI, 2009.

\bibitem{buzios}
G.~Lawler.
\newblock Fractal and multifractal properties of {SLE}, 2010.
\newblock preprint.

\bibitem{Lnew}
G~Lawler.
\newblock Continuity of radial and two-sided radial {SLE} at the terminal
  point, 2011.
\newblock arXiv: 1104.1620.

\bibitem{loop}
G.~Lawler, O.~Schramm, and W.~Werner.
\newblock Conformal invariance of planar loop-erased random walks and uniform
  spanning trees.
\newblock {\em Ann. Probab.}, 32(1B):939--995, 2004.

\bibitem{oksendal}
B.~{\O}ksendal.
\newblock {\em Stochastic differential equations}.
\newblock Universitext. Springer-Verlag, Berlin, sixth edition, 2003.
\newblock An introduction with applications.

\bibitem{cardy1}
M.~A. Rajabpour and J.~Cardy.
\newblock Discretely holomorphic parafermions in lattice {$Z_N$} models.
\newblock {\em J. Phys. A}, 40(49), 2007.

\bibitem{cardy0}
V.~Riva and J.~Cardy.
\newblock Holomorphic parafermions in the {P}otts model and stochastic
  {L}oewner evolution.
\newblock {\em J. Stat. Mech. Theory Exp.}, 2006.

\bibitem{RS}
S.~Rohde and O.~Schramm.
\newblock Basic properties of {SLE}.
\newblock {\em Ann. of Math. (2)}, 161(2):883--924, 2005.

\bibitem{first}
O.~Schramm.
\newblock Scaling limits of loop-erased random walks and uniform spanning
  trees.
\newblock {\em Israel J. Math.}, 118:221--288, 2000.

\bibitem{towards}
S.~Smirnov.
\newblock Towards conformal invariance of 2{D} lattice models.
\newblock In {\em International {C}ongress of {M}athematicians. {V}ol. {II}},
  pages 1421--1451. Eur. Math. Soc., Z\"urich, 2006.

\bibitem{ising1}
S.~Smirnov.
\newblock Conformal invariance in random cluster models. {I}. {H}olomorphic
  fermions in the {I}sing model.
\newblock {\em Ann. of Math. (2)}, 172(2):1435--1467, 2010.

\bibitem{discrete}
S.~Smirnov.
\newblock Discrete complex analysis and probability.
\newblock In {\em Proceedings of the {I}nternational {C}ongress of
  {M}athematicians, {V}ol. {I} ({H}yderabad 2010)}, pages 595--621. Higher Ed.
  Press, Hyderabad, 2010.

\bibitem{stein}
Elias~M. Stein and Rami Shakarchi.
\newblock {\em Complex analysis}.
\newblock Princeton Lectures in Analysis, II. Princeton University Press,
  Princeton, NJ, 2003.

\bibitem{Werner}
W.~Werner.
\newblock Random planar curves and {S}chramm-{L}oewner evolutions.
\newblock In {\em Lectures on probability theory and statistics}, volume 1840
  of {\em Lecture Notes in Math.}, pages 107--195. Springer, Berlin, 2004.

\end{thebibliography}
\end{document}